\documentclass{article}
\usepackage{amsfonts}
\usepackage{amsmath}

\newtheorem{theorem}{Theorem}

\newtheorem{definition}[theorem]{Definition}

\newtheorem{lemma}[theorem]{Lemma}

\newenvironment{proof}[1][Proof]{\noindent\textbf{#1.} }{\ \rule{0.5em}{0.5em}}
\input{tcilatex}
\begin{document}

\title{\textbf{On the growth of meromorphic solutions of homogeneous and
non-homogeneous linear difference equations in terms of (p,q)-order}}
\author{Chinmay Ghosh$^{1}$, Subhadip Khan$^{2}$, Anirban Bandyopadhyay$^{3}$%
\qquad \\
$^{1}$Department of Mathematics\\
Kazi Nazrul University\\
Nazrul Road, P.O.- Kalla C.H.\\
Asansol-713340, West Bengal, India \\
chinmayarp@gmail.com \\
$^{2}$Jotekamal High School\\
Raghunathganj 2, Murshidabad\\
Pin-742133, West Bengal, India\\
subhadip204@gmail.com\\
$^{3}$53, Gopalpur Primary School\\
Raninagar-I, Murshidabad\\
Pin-742304, West Bengal, India \\
anirbanbanerjee010@gmail.com}
\date{}
\maketitle

\begin{abstract}
In this paper we have studied the growth of meromorphic solutions of higher
order homogeneous and non-homogeneous linear difference equations with
entire and meromorphic coefficients. We have extended and improved some
results of Zhou and Zheng $\left( 2017\right) ,$ Belaidi and Benkarouba $%
\left( 2019\right) $ by using $(p,q)-$order and $(p,q)-$type.


\end{abstract}

\section{Introduction}

Recently the properties of meromorphic solutions of complex difference
equations 
\begin{equation}
A_{k}(z)f(z+c_{k})+A_{k-1}(z)f(z+c_{k-1})+\cdots
+A_{1}(z)f(z+c_{1})+A_{0}(z)f(z)=0  \label{1h}
\end{equation}%
and%
\begin{equation}
A_{k}(z)f(z+c_{k})+A_{k-1}(z)f(z+c_{k-1})+\cdots
+A_{1}(z)f(z+c_{1})+A_{0}(z)f(z)=F(z)  \label{1nh}
\end{equation}%
have become a subject of great interest from the view point of Nevanlinna's
theory and achieved many valuable results where the coefficients $%
A_{0},A_{1},\ldots ,A_{k}\neq 0\ $and\ $F\neq 0$ $(k\geq 2)$ in (\ref{1h})
or (\ref{1nh}) are entire or meromorphic functions and $c_{k},c_{k-1},\ldots
,c_{1}$ are distinct nonzero complex numbers.

In $1976$ Juneja, Kapoor and Bajpai \cite{7} introduced the idea of $(p,q)-$%
order of an entire function and in $2010$ Liu, Tu and Shi \cite{10} modified
the definition of the $(p,q)-$order to make it more suitable. Laine and
Yang, in $2007,$ considered the equation\ (\ref{1h}) when more than one
dominant coefficients exist but exactly one has its type strictly greater
than the others (\cite{8}, Theorem 5.2.). In $2008$ Chiang and Feng \cite{3}
investigated meromorphic solutions of (\ref{1h}) and established a theorem (%
\cite{3}, Theorem 9.2) taking exactly one coefficient of (\ref{1h}) with
maximal order$.$ In $2013$, Liu and Mao used hyper order to establish the
case when one or more coefficients of (\ref{1h}) or (\ref{1nh}) having
infinite order (\cite{9}, Theorem 1.4, Theorem 1.6). Finally in $2017$, Zhou
and Zheng (\cite{11}, Theorem 1.5) and in $2019$, Bela\"{\i}di and
Benkarouba (\cite{1}, Theorem 1.1-Theorem 1.4) used iterated order and
iterated type to investigate the solutions of (\ref{1h}) or (\ref{1nh}) and
obtained some results which improved and generalized those previous results.

In this article we use the concept of $(p,q)-$order to investigate
meromorphic solutions of (\ref{1h}) and (\ref{1nh}). We also extend and
improve some results of Zhou and Zheng \cite{11}, Bela\"{\i}di and
Benkarouba \cite{1}. Here we consider both cases, when (\ref{1h}) and (\ref%
{1nh}) have entire coefficients and meromorphic coefficients. We also cover
the cases when, either one of the coefficients have maximal $(p,q)-$order or
more than one coefficients having maximal $(p,q)-$order.

Throughout this paper, we assume that the reader is familiar with the
fundamental results and the standard notations of Nevanlinna's value
distribution theory \cite{6a}.

\section{Definitions and Lemmas}

In this section we give some basic definitions and lemmas which are used to
prove our main results.

For all $r\in \mathbb{R},$ set $\exp _{1}r=e^{r}$ and $\exp _{p+1}r=\exp
\left( \exp _{p}r\right) ,$ $p\in \mathbb{N}.$ Also for all sufficiently
large values of $r,$ $\log _{1}r=\log r$ and $\log _{p+1}r=\log \left( \log
_{p}r\right) ,$ $p\in \mathbb{N}.$ Further $\exp _{0}r=\log _{0}r=r,$ $\exp
_{-1}r=\log _{1}r,$ $\exp _{1}r=\log _{-1}r.$

\begin{definition}
\cite{8a} Let $p\geq q\geq 1$ or $2\leq q=p+1$ be integers. The $(p,q)-$%
order of a transcendental meromorphic function $f$ is defined by%
\begin{equation*}
\rho _{f}\left( p,q\right) =\limsup_{r\rightarrow \infty }\frac{\log
_{p}T(r,f)}{\log _{q}r}.
\end{equation*}

And if $f$ is a transcendental entire function, then%
\begin{equation*}
\rho _{f}\left( p,q\right) =\limsup_{r\rightarrow \infty }\frac{\log
_{p+1}M(r,f)}{\log _{q}r}.
\end{equation*}
\end{definition}

Note that $0\leq \rho _{f}\left( p,q\right) \leq \infty .$ Also for a
rational function $\rho _{f}\left( p,q\right) =0.$

\begin{definition}
\cite{8a} A transcendental meromorphic function is said to have index pair $%
\left[ p,q\right] $ if $0\leq \rho _{f}\left( p,q\right) \leq \infty $ and $%
\rho _{f}\left( p-1,q-1\right) $ is not a non-zero finite number.
\end{definition}

\begin{definition}
\cite{8a} The $(p,q)-$type of a meromorphic function $f$ having non-zero
finite $(p,q)-$order $\rho _{f}\left( p,q\right) $ is defined by%
\begin{equation*}
\tau _{f}\left( p,q\right) =\limsup_{r\rightarrow \infty }\frac{\log
_{p-1}T(r,f)}{\left( \log _{q-1}r\right) ^{\rho _{f}\left( p,q\right) }}.
\end{equation*}

And if $f$ is a transcendental entire function, then%
\begin{equation*}
\tau _{f}\left( p,q\right) =\limsup_{r\rightarrow \infty }\frac{\log
_{p}M(r,f)}{\left( \log _{q-1}r\right) ^{\rho _{f}\left( p,q\right) }}.
\end{equation*}
\end{definition}

\begin{definition}
\cite{8a} Let $p\geq q\geq 1$ or $2\leq q=p+1$ be integers. The $(p,q)-$%
exponent of convergence of the sequence of poles of a meromorphic function $%
f $ is defined by%
\begin{equation*}
\lambda _{\frac{1}{f}}\left( p,q\right) =\limsup_{r\rightarrow \infty }\frac{%
\log _{p}N(r,f)}{\log _{q}r}.
\end{equation*}
\end{definition}

Now we recall that the linear measure of a set $S\subset \left( 0,+\infty
\right) $ is defined by 
\begin{equation*}
m\left( S\right) =\dint\limits_{0}^{\infty }\chi _{S}\left( t\right) dt
\end{equation*}%
and the logarithmic measure of a set $S\subset \left( 1,+\infty \right) $ is
defined by 
\begin{equation*}
lm\left( S\right) =\dint\limits_{1}^{\infty }\frac{\chi _{S}\left( t\right) 
}{t}dt,
\end{equation*}%
where $\chi _{S}\left( t\right) $ is the characteristic function of a set $%
S. $

The upper density of a set $S\subset \left( 0,+\infty \right) $ is defined by%
\begin{equation*}
\overline{dens}S=\limsup\limits_{r\rightarrow \infty }\frac{m\left( S\cap %
\left[ 0,r\right] \right) }{r}
\end{equation*}%
and the upper logarithmic density of a set $S\subset \left( 1,+\infty
\right) $ is defined by%
\begin{equation*}
\overline{\log dens}S=\limsup\limits_{r\rightarrow \infty }\frac{lm\left(
S\cap \left[ 1,r\right] \right) }{\log r}.
\end{equation*}

\begin{lemma}
\label{L1.1}\cite{3} Let $f$ be a meromorphic function, $\xi $ a nonzero
complex number, and let $\nu >1,$ and $\varepsilon >0$ be given real
constants. Then there exists a subset $S\subset \left( 1,+\infty \right) $
of finite logarithmic measure, and a constant $K$ depending only on $\nu $
and $\xi ,$ such that for all $z$ with $\left\vert z\right\vert =r\notin
S\cup \left[ 0,1\right] ,$ we have%
\begin{equation*}
\left\vert \log \left\vert \frac{f\left( z+\xi \right) }{f\left( z\right) }%
\right\vert \right\vert \leq K\left( \frac{T\left( \nu r,f\right) }{r}+\frac{%
n\left( \nu r\right) }{r}\log ^{\nu }r\log ^{+}n\left( \nu r\right) \right) ,
\end{equation*}%
where $n(t)=n(t,f)+n(t,\frac{1}{f}).$
\end{lemma}

\begin{lemma}
\label{L1.2}\cite{5} Let $f$ be a transcendental meromorphic function. Let $%
j $ be a nonnegative integer and $\xi $ be an extended complex number. Then
for a real constant $\alpha >1,$ there exists a constant $R>0,$ such that
for all $r>R,$ we have%
\begin{equation}
n(r,\xi ,f^{(j)})\leq \frac{2j+6}{\log \alpha }T(\alpha r,f).  \label{1.1}
\end{equation}
\end{lemma}

\begin{lemma}
\label{L1.3} Let $f$ be a meromorphic function with finite $\left(
p,q\right) $-order, $\rho _{f}\left( p,q\right) =\rho .$ Let $\xi $ be a
nonzero complex number and $\varepsilon >0$ be given real constant. Then
there exists a subset $S\subset \left( 1,+\infty \right) $ of finite
logarithmic measure such that for all $z$ with $\left\vert z\right\vert
=r\notin S\cup \left[ 0,1\right] ,$ we have 
\begin{equation}
i)\exp \left\{ -r^{\rho -1+\varepsilon }\right\} \leq \left\vert \frac{%
f\left( z+\xi \right) }{f\left( z\right) }\right\vert \leq \exp \left\{
r^{\rho -1+\varepsilon }\right\} ,  \label{1.2a}
\end{equation}%
for $p=q=1,$ and%
\begin{equation}
ii)\exp _{p}\left\{ -\left( \log _{q-1}r\right) ^{\rho +\varepsilon
}\right\} \leq \left\vert \frac{f\left( z+\xi \right) }{f\left( z\right) }%
\right\vert \leq \exp _{p}\left\{ \left( \log _{q-1}r\right) ^{\rho
+\varepsilon }\right\} ,  \label{1.3}
\end{equation}%
for $p\geq q\geq 2.$
\end{lemma}

\begin{proof}
We prove only second part of the lemma. First part follows from \cite{3}.

By Lemma $\ref{L1.1}$, there exists a subset $S\subset \left( 1,+\infty
\right) $ of finite logarithmic measure, and a constant $K$ depending only
on $\nu $ and $\xi ,$ such that for all $z$ with $\left\vert z\right\vert
=r\notin S\cup \left[ 0,1\right] ,$ we have%
\begin{equation}
\left\vert \log \left\vert \frac{f\left( z+\xi \right) }{f\left( z\right) }%
\right\vert \right\vert \leq K\left( \frac{T\left( \nu r,f\right) }{r}+\frac{%
n\left( \nu r\right) }{r}\log ^{\nu }r\log ^{+}n\left( \nu r\right) \right) ,
\label{1.4}
\end{equation}%
where $n(t)=n(t,f)+n(t,\frac{1}{f}).$ Now using $\left( \ref{1.1}\right) $
in $\left( \ref{1.4}\right) $, we obtain%
\begin{eqnarray}
\left\vert \log \left\vert \frac{f\left( z+\xi \right) }{f\left( z\right) }%
\right\vert \right\vert &\leq &K\left( \frac{T\left( \nu r,f\right) }{r}+%
\frac{12}{\log \alpha }\frac{T\left( \alpha \nu r,f\right) }{r}\log ^{\nu
}r\log ^{+}\left( \frac{12}{\log \alpha }T\left( \alpha \nu r,f\right)
\right) \right)  \notag \\
&\leq &K_{1}\left( T\left( \eta r,f\right) \frac{\log ^{\eta }r}{r}\log
T\left( \eta r,f\right) \right) ,  \label{1.5}
\end{eqnarray}%
where $K_{1}>0$ is some constant and we consider $\eta =\alpha \nu >1.$

Now since $f$ has finite $\left( p,q\right) $-order $\rho _{f}\left(
p,q\right) =\rho ,$ so for given $\varepsilon ,$ $0<\varepsilon <2,$ and for
sufficiently large $r$ we have%
\begin{equation}
T\left( r,f\right) \leq \exp _{p-1}\left\{ \left( \log _{q-1}r\right) ^{\rho
+\frac{\varepsilon }{2}}\right\}  \label{1.6}
\end{equation}

Therefore using $\left( \ref{1.6}\right) $ in $\left( \ref{1.5}\right) $, we
obtain%
\begin{eqnarray}
\left\vert \log \left\vert \frac{f\left( z+\xi \right) }{f\left( z\right) }%
\right\vert \right\vert &\leq &K_{1}\exp _{p-1}\left\{ \log _{q-1}\left(
\eta r\right) ^{\rho +\frac{\varepsilon }{2}}\right\} \frac{\log ^{\eta }r}{r%
}\exp _{p-2}\left\{ \left( \log _{q-1}\left( \eta r\right) \right) ^{\rho +%
\frac{\varepsilon }{2}}\right\}  \notag \\
&\leq &\exp _{p-1}\left\{ \left( \log _{q-1}r\right) ^{\rho +\varepsilon
}\right\} .  \label{1.7}
\end{eqnarray}

Hence from $\left( \ref{1.7}\right) $ we obtain the required result.
\end{proof}

\begin{lemma}
\label{L1.4} Let $\xi _{1}$ and $\xi _{2}$ be two arbitrary distinct complex
numbers and $f$ be a meromorphic function with finite $\left( p,q\right) $%
-order, $\rho _{f}\left( p,q\right) =\rho .$ Then for given $\varepsilon >0,$
there exists a subset $S\subset \left( 1,+\infty \right) $ of finite
logarithmic measure such that for all $z$ with $\left\vert z\right\vert
=r\notin S\cup \left[ 0,1\right] ,$ we have%
\begin{equation*}
i)\text{ }\exp \left\{ -r^{\rho -1+\varepsilon }\right\} \leq \left\vert 
\frac{f\left( z+\xi _{1}\right) }{f\left( z+\xi _{2}\right) }\right\vert
\leq \exp \left\{ r^{\rho -1+\varepsilon }\right\} ,
\end{equation*}%
for $p=q=1,$ and%
\begin{equation*}
ii)\text{ }\exp _{p}\left\{ -\left( \log _{q-1}r\right) ^{\rho +\varepsilon
}\right\} \leq \left\vert \frac{f\left( z+\xi _{1}\right) }{f\left( z+\xi
_{2}\right) }\right\vert \leq \exp _{p}\left\{ \left( \log _{q-1}r\right)
^{\rho +\varepsilon }\right\} ,
\end{equation*}%
for $p\geq q\geq 2.$
\end{lemma}

\begin{proof}
We prove only second part of the lemma. First part follows from \cite{3}.

For the second part we consider the following expression%
\begin{equation*}
\left\vert \frac{f\left( z+\xi _{1}\right) }{f\left( z+\xi _{2}\right) }%
\right\vert =\left\vert \frac{f\left( z+\xi _{2}+\xi _{1}-\xi _{2}\right) }{%
f\left( z+\xi _{2}\right) }\right\vert ,\text{ }\left( \xi _{1}\neq \xi
_{2}\right) .
\end{equation*}

Now by using Lemma $\ref{L1.3}$, for any given $\varepsilon >0$ and for all $%
z$ with $\left\vert z+\xi _{2}\right\vert =R\notin S\cup \left[ 0,1\right] $
such that $lm\left( S\right) <\infty ,$ we have%
\begin{eqnarray*}
\exp _{p}\left\{ -\left( \log _{q-1}\left( r\right) \right) ^{\rho
+\varepsilon }\right\} &\leq &\exp _{p}\left\{ -\left( \log _{q-1}\left(
\left\vert z\right\vert +\xi _{2}\right) \right) ^{\rho +\frac{\varepsilon }{%
2}}\right\} \\
&\leq &\exp _{p}\left\{ -\left( \log _{q-1}R\right) ^{\rho +\frac{%
\varepsilon }{2}}\right\} \\
&\leq &\left\vert \frac{f\left( z+\xi _{1}\right) }{f\left( z+\xi
_{2}\right) }\right\vert =\left\vert \frac{f\left( z+\xi _{2}+\xi _{1}-\xi
_{2}\right) }{f\left( z+\xi _{2}\right) }\right\vert \\
&\leq &\exp _{p}\left\{ \left( \log _{q-1}R\right) ^{\rho +\frac{\varepsilon 
}{2}}\right\} \leq \exp _{p}\left\{ \left( \log _{q-1}\left( \left\vert
z\right\vert +\left\vert \xi _{2}\right\vert \right) \right) ^{\rho +\frac{%
\varepsilon }{2}}\right\} \\
&\leq &\exp _{p}\left\{ \left( \log _{q-1}\left( r\right) \right) ^{\rho
+\varepsilon }\right\} ,
\end{eqnarray*}%
for $z$ with $\left\vert z\right\vert =r\notin S\cup \left[ 0,1\right] $,
where $S\subset \left( 1,+\infty \right) $ is a set of finite logarithmic
measure.

Hence we obtain the required result.
\end{proof}

\begin{lemma}
\label{L1.5}\cite{6} Let $f$ be a nonconstant meromorphic function. Suppose $%
z_{1}\in \mathbb{C},$ $\delta <1~$and $\varepsilon >0.$ Then for all $r$
outside of a possible exceptional set $S$ with finite logarithmic measure $%
\dint\limits_{S}\frac{dr}{r}<\infty ,$ we have 
\begin{equation*}
m\left( r,\frac{f\left( z+z_{1}\right) }{f\left( z\right) }\right) =o\left( 
\frac{\left( T\left( r+\left\vert z_{1}\right\vert ,\text{ }f\right) \right)
^{1+\varepsilon }}{r^{\delta }}\right) .
\end{equation*}
\end{lemma}

\begin{lemma}
\label{L1.5A}\cite{4} Let $f$ be a nonconstant meromorphic function and $%
z_{1},$ $z_{2}$ be nonzero complex constants. Then for $r\rightarrow +\infty 
$ we have%
\begin{equation*}
\left( 1+o(1)\right) T\left( r-\left\vert z_{1}\right\vert ,\text{ }f\right)
\leq T\left( r,\text{ }f\left( z+z_{1}\right) \right) \leq \left(
1+o(1)\right) T\left( r+\left\vert z_{1}\right\vert ,\text{ }f\right) .
\end{equation*}

Consequently, 
\begin{equation*}
\rho _{f\left( z+z_{2}\right) }\left( p,q\right) =\rho _{f}\left( p,q\right)
,
\end{equation*}
for $p\geq q,$ $p,q\in \mathbb{N}.$
\end{lemma}

\begin{lemma}
\label{L1.6}\cite{6} Let $f$ be a nonconstant meromorphic function. Suppose $%
z_{1},z_{2}\in \mathbb{C},$ such that $z_{1}\neq z_{2},$ $\delta <1,$ $%
\varepsilon >0.$ Then 
\begin{equation*}
m\left( r,\frac{f\left( z+z_{1}\right) }{f\left( z+z_{2}\right) }\right)
=o\left( \frac{\left\{ T\left( r+\left\vert z_{1}-z_{2}\right\vert
+\left\vert z_{2}\right\vert ,\text{ }f\right) \right\} ^{1+\varepsilon }}{%
r^{\delta }}\right)
\end{equation*}%
for all $r$ outside of a possible exceptional set $S$ with finite
logarithmic measure $\dint\limits_{S}\frac{dr}{r}<\infty .$
\end{lemma}

\begin{lemma}
\label{L1.7}\cite{8a} Let $f$ be a nonconstant meromorphic function with
nonzero finite $(p,q)$-order $\rho _{f}\left( p,q\right) $ and nonzero
finite $(p,q)$-type $\tau _{f}\left( p,q\right) .$ Then for any given $%
b<\tau _{f}\left( p,q\right) ,$ there exists a subset $S\subset \left[
1,+\infty \right) $ of infinite logarithmic measure such that for all $r\in
S,$ we have 
\begin{equation*}
\log _{p-1}T\left( r,\text{ }f\right) >b\left( \log _{q-1}r\right) ^{\rho
_{f}\left( p,q\right) }.
\end{equation*}
\end{lemma}

\begin{lemma}
\label{L1.8}\cite{3} Let $\alpha ~,$ $R,~R^{\prime }$ be real numbers such
that $0<\alpha <1,~R,~R^{\prime }>0$ and let $\xi $ be a nonzero complex
number. Then there exists a positive constant $K_{\alpha }$ which depends
only on $\alpha $ such that for a given meromorphic function $f$ we have,
when $\left\vert z\right\vert =r,$ $\max \left\{ 1,r+\left\vert \xi
\right\vert \right\} <R<R^{\prime },$ the estimate%
\begin{gather*}
m\left( r,\frac{f\left( z+\xi \right) }{f\left( z\right) }\right) +m\left( r,%
\frac{f\left( z\right) }{f\left( z+\xi \right) }\right) \leq \frac{%
2\left\vert \xi \right\vert R}{\left( R-r-\left\vert \xi \right\vert \right)
^{2}}\left( m(R,\text{ }f)+m\left( R,\frac{1}{f}\right) \right) \\
+\frac{2R^{\prime }}{(R^{\prime }-R)}\left( \frac{\left\vert \xi \right\vert 
}{R-r-\left\vert \xi \right\vert }+\frac{K_{\alpha }\left\vert \xi
\right\vert ^{\alpha }}{\left( 1-\alpha \right) r^{\alpha }}\right) \left(
N\left( R^{\prime },\text{ }f\right) +N\left( R^{\prime },\frac{1}{f}\right)
\right) .
\end{gather*}
\end{lemma}

\begin{lemma}
\label{L1.9} Let $\xi _{1},\xi _{2}$ be two complex numbers such that $\xi
_{1}\neq \xi _{2}$ and suppose $f$ be of finite $(p,q)$-order meromorphic
function. Consider the $(p,q)$-order of $f$ as $\rho _{f}\left( p,q\right)
=\rho <+\infty .$ Then for each $\varepsilon >0,$ we have%
\begin{equation*}
i)\text{ }m\left( r,\frac{f\left( z+\xi _{1}\right) }{f\left( z+\xi
_{2}\right) }\right) =O\left( r^{\rho -1+\varepsilon }\right) ,
\end{equation*}%
for $p=q=1,$ and%
\begin{equation*}
ii)\text{ }m\left( r,\frac{f\left( z+\xi _{1}\right) }{f\left( z+\xi
_{2}\right) }\right) =O\left( \exp _{p-1}\left[ \left\{ \log _{q-1}\left(
r\right) \right\} ^{\rho +\varepsilon }\right] \right) ,
\end{equation*}%
for $p\geq q\geq 2.$
\end{lemma}

\begin{proof}
We prove only second part of the lemma. First part follows from \cite{3}.

In the second part, for $p\geq q\geq 2$ we first write the following
expression as%
\begin{eqnarray}
m\left( r,\frac{f\left( z+\xi _{1}\right) }{f\left( z+\xi _{2}\right) }%
\right) &\leq &m\left( r,\frac{f\left( z+\xi _{1}\right) }{f\left( z\right) }%
\right) +m\left( r,\frac{f\left( z\right) }{f\left( z+\xi _{2}\right) }%
\right)  \notag \\
&\leq &m\left( r,\frac{f\left( z+\xi _{1}\right) }{f\left( z\right) }\right)
+m\left( r,\frac{f\left( z\right) }{f\left( z+\xi _{1}\right) }\right) 
\notag \\
&&+m\left( r,\frac{f\left( z\right) }{f\left( z+\xi _{2}\right) }\right)
+m\left( r,\frac{f\left( z+\xi _{2}\right) }{f\left( z\right) }\right) .
\label{1.8}
\end{eqnarray}

By Lemma $\ref{L1.8}$ and using the concept given in [\cite{1}, lemma 2.9],
we obtain\ from above%
\begin{equation}
m\left( r,\frac{f\left( z+\xi _{1}\right) }{f\left( z+\xi _{2}\right) }%
\right) \leq 4\left[ 
\begin{array}{c}
\frac{4\left\vert \xi _{1}\right\vert r}{\left( r-\left\vert \xi
_{1}\right\vert ^{2}\right) }+\frac{4\left\vert \xi _{2}\right\vert r}{%
\left( r-\left\vert \xi _{2}\right\vert ^{2}\right) } \\ 
+6\left( \frac{\left\vert \xi _{1}\right\vert }{\left( r-\left\vert \xi
_{1}\right\vert \right) }+\frac{\left\vert \xi _{2}\right\vert }{\left(
r-\left\vert \xi _{2}\right\vert \right) }\right) +\frac{2K_{\alpha }\left(
\left\vert \xi _{1}\right\vert ^{1-\frac{\varepsilon }{2}}+\left\vert \xi
_{2}\right\vert ^{1-\frac{\varepsilon }{2}}\right) }{\varepsilon r^{1-\frac{%
\varepsilon }{2}}}%
\end{array}%
\right] T(3r,f).  \label{1.9a}
\end{equation}

Now since the $(p,q)$-order of $f$ is $\rho _{f}\left( p,q\right) =\rho
<+\infty ,$ so given $0<\varepsilon <2,$ by definition we have%
\begin{equation*}
T\left( r,f\right) \leq \exp _{p-1}\left\{ \left\{ \log _{q-1}\left(
r\right) \right\} ^{\rho +\frac{\varepsilon }{2}}\right\} .
\end{equation*}

Using the above in $\left( \ref{1.9a}\right) $ we obtain%
\begin{eqnarray*}
&&m\left( r,\frac{f\left( z+\xi _{1}\right) }{f\left( z+\xi _{2}\right) }%
\right) \\
&\leq &4\left[ 
\begin{array}{c}
\frac{4\left\vert \xi _{1}\right\vert r}{\left( r-\left\vert \xi
_{1}\right\vert ^{2}\right) }+\frac{4\left\vert \xi _{2}\right\vert r}{%
\left( r-\left\vert \xi _{2}\right\vert ^{2}\right) } \\ 
+6\left( \frac{\left\vert \xi _{1}\right\vert }{\left( r-\left\vert \xi
_{1}\right\vert \right) }+\frac{\left\vert \xi _{2}\right\vert }{\left(
r-\left\vert \xi _{2}\right\vert \right) }\right) +\frac{2K_{\alpha }\left(
\left\vert \xi _{1}\right\vert ^{1-\frac{\varepsilon }{2}}+\left\vert \xi
_{2}\right\vert ^{1-\frac{\varepsilon }{2}}\right) }{\varepsilon r^{1-\frac{%
\varepsilon }{2}}}%
\end{array}%
\right] \exp _{p-1}\left[ \left\{ \log _{q-1}\left( 3r\right) \right\}
^{\rho +\frac{\varepsilon }{2}}\right] \\
&\leq &K\exp _{p-1}\left\{ \log _{q-1}r^{\rho +\varepsilon }\right\} ,\text{
where }K>0\text{ is a constant.}
\end{eqnarray*}

This completes the proof.
\end{proof}

\begin{lemma}
\label{L1.10} Let $D$ be a complex set satisfying $\overline{\log dens}%
\left\{ r=\left\vert z\right\vert :z\in D\right\} >0$ and let $%
A_{0}(z),A_{1}(z),....,A_{k}(z)$ be entire functions of $\left( p,q\right) $%
-order satisfying $\max\limits_{0\leq j\leq k}\left\{ \rho _{A_{j}}\left(
p,q\right) \right\} \leq \rho .$ If there exists an integer $l$ $\left(
0\leq l\leq k\right) $ such that for some constants $a,b$ $(0\leq b<a)$ and $%
\delta $ $\left( 0<\delta <\rho \right) $ sufficiently small with 
\begin{equation}
\left\vert A_{l}(z)\right\vert \geq \exp _{p}\left[ a\left\{ \log
_{q-1}\left( r\right) \right\} ^{\rho -\delta }\right]  \label{1.10}
\end{equation}%
\begin{equation*}
\text{ and }\left\vert A_{j}(z)\right\vert \leq \exp _{p}\left[ b\left\{
\log _{q-1}\left( r\right) \right\} ^{\rho -\delta }\right] ,\text{ }%
j=0,1,...,k,\text{ }j\neq l
\end{equation*}%
as $z\rightarrow \infty $ for $z\in D,$ then we have $\rho _{A_{l}}\left(
p,q\right) =\rho .$
\end{lemma}

\begin{proof}
By the stated condition we have $\rho _{A_{l}}\left( p,q\right) \leq \rho .$
Let $\rho _{A_{l}}\left( p,q\right) =\alpha <\rho .$

Then for given $\varepsilon $ and sufficiently large $r,$ by definition we
have%
\begin{equation*}
\left\vert A_{l}(z)\right\vert \leq \exp _{p}\left[ \left\{ \log
_{q-1}\left( r\right) \right\} ^{\alpha +\varepsilon }\right]
\end{equation*}

Again by $\left( \ref{1.10}\right) $,%
\begin{equation*}
\left\vert A_{l}(z)\right\vert \geq \exp _{p}\left[ a\left\{ \log
_{q-1}\left( r\right) \right\} ^{\rho -\delta }\right] .
\end{equation*}

Combining the above two for $z\in D,$ $\left\vert z\right\vert \rightarrow
+\infty ,$ we obtain%
\begin{equation*}
\exp _{p}\left[ a\left\{ \log _{q-1}\left( r\right) \right\} ^{\rho -\delta }%
\right] \leq \left\vert A_{l}(z)\right\vert \leq \exp _{p}\left[ \left\{
\log _{q-1}\left( r\right) \right\} ^{\alpha +\varepsilon }\right] ,
\end{equation*}%
where $\varepsilon $ is arbitrary and $0<\varepsilon <\rho -\alpha -2\delta
, $ which is a contradiction as $r\rightarrow +\infty .$

Hence $\rho _{A_{l}}\left( p,q\right) =\rho .$
\end{proof}

\begin{lemma}
\label{L1.11} Let $D$ be a complex set satisfying $\overline{\log dens}%
\left\{ r=\left\vert z\right\vert :z\in D\right\} >0$ and let $%
A_{0}(z),A_{1}(z),....,A_{k}(z)$ be entire functions of $\left( p,q\right) -$%
order satisfying $\max\limits_{0\leq j\leq k}\left\{ \rho _{A_{j}}\left(
p,q\right) \right\} \leq \rho .$ If there exists an integer $l$ $\left(
0\leq l\leq k\right) $ such that for some constants $a,b$ $(0\leq b<a)$ and $%
\delta $ $\left( 0<\delta <\rho \right) $ sufficiently small with 
\begin{equation*}
T(r,A_{l})\geq \exp _{p-1}\left[ a\left\{ \log _{q-1}\left( r\right)
\right\} ^{\rho -\delta }\right]
\end{equation*}%
\begin{equation*}
\text{and }T(r,A_{j})\leq \exp _{p-1}\left[ b\left\{ \log _{q-1}\left(
r\right) \right\} ^{\rho -\delta }\right] ,\text{ }j=0,1,...,k,\text{ }j\neq
l
\end{equation*}%
as $z\rightarrow \infty $ for $z\in D,$ then we have $\rho _{A_{l}}\left(
p,q\right) =\rho .$
\end{lemma}

\begin{proof}
The proof follows from the previous lemma, hence we omit it.
\end{proof}

\section{Main Results}

In this section we establish our main results.

\begin{theorem}
\label{T2.1} Let $D$ be a complex set satisfying $\overline{\log dens}%
\left\{ r=\left\vert z\right\vert :z\in D\right\} >0$ and let $%
A_{0}(z),A_{1}(z),....,A_{k}(z)$ be entire functions of $\left( p,q\right) $%
-order satisfying $\max\limits_{0\leq j\leq k}\left\{ \rho _{A_{j}}\left(
p,q\right) \right\} \leq \rho .$ If there exists an integer $l$ $\left(
0\leq l\leq k\right) $ such that for some constants $a,b$ $(0\leq b<a)$ and $%
\delta $ $\left( 0<\delta <\rho \right) $ sufficiently small with 
\begin{equation}
\left\vert A_{l}(z)\right\vert \geq \exp _{p}\left[ a\left\{ \log
_{q-1}\left( r\right) \right\} ^{\rho -\delta }\right]  \label{2.1}
\end{equation}%
\begin{equation}
\left\vert A_{j}(z)\right\vert \leq \exp _{p}\left[ b\left\{ \log
_{q-1}\left( r\right) \right\} ^{\rho -\delta }\right] ,\text{ }j=0,1,...,k,%
\text{ }j\neq l  \label{2.2}
\end{equation}%
as $z\rightarrow \infty $ for $z\in D,$ then every meromorphic solution $%
f\left( \neq 0\right) $ of equation $\left( \ref{1h}\right) $ satisfies

$\left( i\right) $ $\rho _{f}\geq \rho _{A_{l}}+1,$ for $p=1,$ $q=1.$

$\left( ii\right) $ $\rho _{f}\left( p,q\right) \geq \rho _{A_{l}}\left(
p,q\right) ,$ for $p\geq q\geq 2.$
\end{theorem}

\begin{proof}
For $p=q=1,$ see \cite{1}. We consider the case when $p\geq q\geq 2.$

First let $f$ $\left( \neq 0\right) $ be a meromorphic solution of $\left( %
\ref{1h}\right) $ and if possible let $\rho _{f}\left( p,q\right) <\rho .$

Now divide $\left( \ref{1h}\right) $ by $f(z+c_{l})$ we get%
\begin{equation}
-A_{l}(z)=A_{k}(z)\frac{f(z+c_{k})}{f(z+c_{l})}+...+A_{l-1}(z)\frac{%
f(z+c_{l-1})}{f(z+c_{l})}+...+A_{1}(z)\frac{f(z+c_{1})}{f(z+c_{l})}+A_{0}(z)%
\frac{f(z)}{f(z+c_{l})}.  \label{2.3}
\end{equation}

The above expression can be written as%
\begin{equation*}
-1=\dsum\limits_{j=1,j\neq l}^{k}\frac{A_{j}(z)f(z+c_{j})}{A_{l}(z)f(z+c_{l})%
}+\frac{A_{0}(z)f(z)}{A_{l}(z)f(z+c_{l})}.
\end{equation*}

The above implies%
\begin{equation}
1\leq \dsum\limits_{j=1,j\neq l}^{k}\left\vert \frac{A_{j}(z)f(z+c_{j})}{%
A_{l}(z)f(z+c_{l})}\right\vert +\left\vert \frac{A_{0}(z)f(z)}{%
A_{l}(z)f(z+c_{l})}\right\vert .  \label{2.3A}
\end{equation}

By Lemma $\ref{L1.4}$ $\left( ii\right) $, for any given $\varepsilon >0$ $%
\left( \varepsilon <\rho -\rho _{f}\left( p,q\right) -2\delta \right) ,$
there exists a subset $S\subset \left( 1,+\infty \right) $ of finite
logarithmic measure such that for all $\left\vert z\right\vert =r\notin
S\cup \left[ 0,1\right] ,$ we have%
\begin{equation}
\left\vert \frac{f(z+c_{j})}{f(z+c_{l})}\right\vert \leq \exp _{p}\left[
\left\{ \log _{q-1}\left( r\right) \right\} ^{\rho _{f}\left( p,q\right)
+\varepsilon }\right] <\exp _{p}\left[ \left\{ \log _{q-1}\left( r\right)
\right\} ^{\rho -2\delta }\right] ,\left( j\neq l,j=1,2,....,k\right)
\label{2.4}
\end{equation}%
and%
\begin{equation}
\left\vert \frac{f(z)}{f(z+c_{l})}\right\vert \leq \exp _{p}\left[ \left\{
\log _{q-1}\left( r\right) \right\} ^{\rho _{f}\left( p,q\right)
+\varepsilon }\right] <\exp _{p}\left[ \left\{ \log _{q-1}\left( r\right)
\right\} ^{\rho -2\delta }\right] .  \label{2.5}
\end{equation}

Now $D$ is a complex set satisfying $\overline{\log dens}\left\{
r=\left\vert z\right\vert :z\in D\right\} >0$ and for $\left\vert
z\right\vert \rightarrow +\infty ,$ we have $\left( \ref{2.1}\right) $ and $%
\left( \ref{2.2}\right) $. Therefore we set $D_{1}=\left\{ r=\left\vert
z\right\vert :z\in D\right\} .$

Since $\overline{\log dens}\left\{ r=\left\vert z\right\vert :z\in D\right\}
>0,$ thus $D_{1}$ is a set of $r$ with $\dint\limits_{D_{1}}\frac{dr}{r}%
=\infty .$

Now for $z\in D_{1}\backslash S\cup \left[ 0,1\right] ,$ substituting $%
\left( \ref{2.1}\right) $,$\left( \ref{2.2}\right) $,$\left( \ref{2.4}%
\right) $ and $\left( \ref{2.5}\right) $ in $\left( \ref{2.3A}\right) $, we
obtain%
\begin{equation*}
1\leq k\frac{\exp _{p}\left[ b\left\{ \log _{q-1}\left( r\right) \right\}
^{\rho -\delta }\right] }{\exp _{p}\left[ a\left\{ \log _{q-1}\left(
r\right) \right\} ^{\rho -\delta }\right] }.\exp _{p}\left[ \left\{ \log
_{q-1}\left( r\right) \right\} ^{\rho -2\delta }\right] \rightarrow 0\text{
as }r\rightarrow \infty .
\end{equation*}

The above expression leads to a contradiction.

Hence we get $\rho _{f}\left( p,q\right) \geq \rho .$

Again by Lemma $\ref{L1.10}$ we know $\rho _{A_{l}}\left( p,q\right) =\rho ,$
hence $\rho _{f}\left( p,q\right) \geq \rho _{A_{l}}\left( p,q\right) .$
\end{proof}

\begin{theorem}
\label{T2.2} Let $D$ be a complex set satisfying $\overline{\log dens}%
\left\{ r=\left\vert z\right\vert :z\in D\right\} >0$ and let $%
A_{0}(z),A_{1}(z),....,A_{k}(z)$ be entire functions satisfying $%
\max\limits_{0\leq j\leq k}\left\{ \rho _{A_{j}}\left( p,q\right) \right\}
\leq \rho .$ If there exists an integer $l$ $\left( 0\leq l\leq k\right) $
such that for some constants $a,b$ $(0\leq b<a)$ and $\delta $ $\left(
0<\delta <\rho \right) $ sufficiently small with 
\begin{equation}
T(r,A_{l})\geq \exp _{p-1}\left[ \left\{ \log _{q-1}\left( r\right) \right\}
^{\rho -\delta }\right]  \label{2.6}
\end{equation}%
\begin{equation}
T(r,A_{j})\leq \exp _{p-1}\left[ \left\{ \log _{q-1}\left( r\right) \right\}
^{\rho -\delta }\right] ,\text{ }j=0,1,...,k,\text{ }j\neq l  \label{2.7}
\end{equation}%
as $z\rightarrow \infty $ for $z\in D,$ then every meromorphic solution $%
f\left( \neq 0\right) $ of equation $\left( \ref{1h}\right) $ satisfies

$\left( i\right) $ $\rho _{f}\geq \rho _{A_{l}}+1,$ for $p=1,$ $q=1$ and $%
0\leq kb<a.$

$\left( ii\right) $ $\rho _{f}\left( p,q\right) \geq \rho _{A_{l}}\left(
p,q\right) ,$ for $p\geq q\geq 2$ \ and $0\leq b<a.$
\end{theorem}

\begin{proof}
For $p=q=1,$ see \cite{1}. We consider the case when $p\geq q\geq 2.$

First let $f$ $\left( \neq 0\right) $ be a meromorphic solution of $\left( %
\ref{1h}\right) $ and if possible let $\rho _{f}\left( p,q\right) <\rho .$

Now since $A_{0}(z),A_{1}(z),....,A_{k}(z)$ are entire, by $\left( \ref{2.3}%
\right) $ we have%
\begin{eqnarray}
m(r,A_{l}) &=&T(r,A_{l})  \notag \\
&\leq &\dsum\limits_{j=0,j\neq l}^{k}m(r,A_{j})+\dsum\limits_{j=1,j\neq
l}^{k}m\left( r,\frac{f(z+c_{j})}{f(z+c_{l})}\right) +m\left( r,\frac{f(z)}{%
f(z+c_{l})}\right) +O(1)  \notag \\
&=&\dsum\limits_{j=0,j\neq l}^{k}T(r,A_{j})+\dsum\limits_{j=1,j\neq
l}^{k}m\left( r,\frac{f(z+c_{j})}{f(z+c_{l})}\right) +m\left( r,\frac{f(z)}{%
f(z+c_{l})}\right) +O(1).  \label{2.7A}
\end{eqnarray}

For any given $\varepsilon $ $\left( 0<\varepsilon <\rho -\rho _{f}\left(
p,q\right) -2\delta \right) ,$ from Lemma $\left( \ref{L1.9}\right) $ the
above implies%
\begin{eqnarray}
T(r,A_{l}) &\leq &\dsum\limits_{j=0,j\neq
l}^{k}T(r,A_{j})+\dsum\limits_{j=1,j\neq l}^{k}\exp _{p-1}\left\{ \left\{
\log _{q-1}\left( r\right) \right\} ^{\rho _{f}\left( p,q\right)
+\varepsilon }\right\}  \label{2.8} \\
&&+\exp _{p-1}\left\{ \left\{ \log _{q-1}\left( r\right) \right\} ^{\rho
_{f}\left( p,q\right) +\varepsilon }\right\} +O(1).  \notag
\end{eqnarray}

Substituting $\left( \ref{2.6}\right) $ and $\left( \ref{2.7}\right) $ in $%
\left( \ref{2.8}\right) $, we obtain%
\begin{gather}
\exp _{p-1}\left[ \left\{ a\log _{q-1}\left( r\right) \right\} ^{\rho
-\delta }\right] \leq \dsum\limits_{j=0,j\neq l}^{k}\exp _{p-1}\left[
b\left\{ \log _{q-1}\left( r\right) \right\} ^{\rho -\delta }\right]  \notag
\\
+\dsum\limits_{j=1,j\neq l}^{k}\exp _{p-1}\left\{ \left\{ \log _{q-1}\left(
r\right) \right\} ^{\rho _{f}\left( p,q\right) +\varepsilon }\right\} +\exp
_{p-1}\left\{ \left\{ \log _{q-1}\left( r\right) \right\} ^{\rho _{f}\left(
p,q\right) +\varepsilon }\right\} +O(1)  \notag \\
\leq k\exp _{p-1}\left[ b\left\{ \log _{q-1}\left( r\right) \right\} ^{\rho
-\delta }\right] +k\exp _{p-1}\left\{ \left\{ \log _{q-1}\left( r\right)
\right\} ^{\rho _{f}\left( p,q\right) +\varepsilon }\right\} +O(1).
\label{2.9}
\end{gather}

By $\left( \ref{2.9}\right) $ it follows%
\begin{equation*}
(a-b)\left\{ \log _{q-1}\left( r\right) \right\} ^{\rho -\delta }\leq
\left\{ \log _{q-1}\left( r\right) \right\} ^{\rho _{f}\left( p,q\right)
+\varepsilon }+O(1).
\end{equation*}

Since $(a-b)>0,$ the above implies%
\begin{equation*}
1\leq \frac{\left\{ \log _{q-1}\left( r\right) \right\} ^{\rho _{f}\left(
p,q\right) +\varepsilon -\rho +\delta }}{(a-b)}+\frac{O(1)}{(a-b)\left\{
\log _{q-1}\left( r\right) \right\} ^{\rho -\delta }}\rightarrow 0\text{ as }%
r\rightarrow +\infty ,
\end{equation*}%
which is a contradiction.

Again by Lemma $\ref{L1.11}$ it follows that $\rho _{A_{l}}\left( p,q\right)
=\rho .$ Hence we have $\rho _{f}\left( p,q\right) \geq \rho _{A_{l}}\left(
p,q\right) $ and the theorem is proved.
\end{proof}

\begin{theorem}
\label{T2.3} Let $A_{0}(z),A_{1}(z),....,A_{k}(z)$ be entire functions and
there exists an integer $l$ $\left( 0\leq l\leq k\right) $ such that 
\begin{equation*}
\max \left\{ \rho _{A_{j}}\left( p,q\right) :j=0,1,...,k,\text{ }j\neq
l\right\} \leq \rho _{A_{l}}\left( p,q\right) ,
\end{equation*}%
\begin{equation*}
\max \left\{ \tau _{A_{j}}(p,q):\rho _{A_{j}}\left( p,q\right) =\rho
_{A_{l}}\left( p,q\right) \right\} <\tau _{A_{l}}(p,q),
\end{equation*}%
where $0<\rho _{A_{l}}\left( p,q\right) ,$ $\tau _{A_{l}}(p,q)<\infty $ and $%
p\geq q\geq 1$ are integers. Then every meromorphic solution $\left( f\neq
0\right) $ of $\left( \ref{1h}\right) $ satisfies $\rho _{f}\left(
p,q\right) \geq \rho _{A_{l}}\left( p,q\right) .$
\end{theorem}

\begin{proof}
Suppose $f$ $\left( \neq 0\right) $ be a meromorphic solution of the
equation $\left( \ref{1h}\right) $.

Now from $\left( \ref{2.7A}\right) $ by using Lemma $\ref{L1.6}$, for all $r$
outside of a possible exceptional set $S_{1}$ with finite logarithmic
measure we obtain (see \cite{1})%
\begin{eqnarray}
m(r,A_{l}) &=&T(r,A_{l})\leq \dsum\limits_{j=0,j\neq
l}^{k}m(r,A_{j})+\dsum\limits_{j=1,j\neq l}^{k}m\left( r,\frac{f(z+c_{j})}{%
f(z+c_{l})}\right) +m\left( r,\frac{f(z)}{f(z+c_{l})}\right) +O(1)  \notag \\
&\leq &\dsum\limits_{j=0,j\neq l}^{k}T(r,A_{j})+\dsum\limits_{j=1,j\neq
l}^{k}o\left( \frac{\left( T\left( r+\left\vert c_{j}-c_{l}\right\vert
+\left\vert c_{l}\right\vert ,f\right) \right) ^{1+\varepsilon }}{r^{\delta }%
}\right) +  \notag \\
&&o\left( \frac{\left( T\left( r+2\left\vert c_{l}\right\vert ,f\right)
\right) ^{1+\varepsilon }}{r^{\delta }}\right) +O(1)  \notag \\
&\leq &\dsum\limits_{j=0,j\neq l}^{k}T(r,A_{j})+o\left( \frac{\left( T\left(
r+2\left\vert c_{l}\right\vert ,f\right) \right) ^{1+\varepsilon }}{%
r^{\delta }}\right) .  \label{2.10}
\end{eqnarray}

Consider two real numbers $b_{1},b_{2}$ such that 
\begin{equation*}
\max \left\{ \tau _{A_{j}}(p,q):\rho _{A_{j}}\left( p,q\right) =\rho
_{A_{l}}\left( p,q\right) \right\} <b_{1}<b_{2}<\tau _{A_{l}}(p,q).
\end{equation*}

Now by Lemma $\ref{L1.7}$ there exists a subset $S_{2}\subset \left[
1,+\infty \right) $ of infinite logarithmic measure such that for all $r\in
S_{2},$ we have 
\begin{equation*}
\log _{p-1}T(r,A_{l})>b_{2}\left( \log _{q-1}r\right) ^{\rho _{A_{l}}\left(
p,q\right) }.
\end{equation*}

Therefore for a sequence $\left\{ r_{n}\right\} $ such that $r_{n}\in
S_{2},r_{n}\rightarrow \infty $ we have%
\begin{equation}
\log _{p-1}T(r_{n},A_{l})>b_{2}\left( \log _{q-1}r_{n}\right) ^{\rho
_{A_{l}}\left( p,q\right) }.  \label{2.11}
\end{equation}

Now if we take $b=\max \left\{ \rho _{A_{j}}\left( p,q\right) :j=0,1,...,k,%
\text{ }j\neq l\right\} <\rho _{A_{l}}\left( p,q\right) ,$ then for any
given $\varepsilon $ $\left( 0<\varepsilon <\rho _{A_{l}}\left( p,q\right)
-b\right) $ and sufficiently large $r_{n}$ , we have%
\begin{equation}
T(r_{n},A_{j})\leq \exp _{p-1}\left\{ \left( \log _{q-1}r_{n}\right)
^{b+\varepsilon }\right\} \leq \exp _{p-1}\left\{ b_{1}\left( \log
_{q-1}r_{n}\right) ^{\rho _{A_{l}}\left( p,q\right) }\right\} .  \label{2.12}
\end{equation}

Again since%
\begin{equation*}
\max \left\{ \tau _{A_{j}}(p,q):\rho _{A_{j}}\left( p,q\right) =\rho
_{A_{l}}\left( p,q\right) \right\} <\tau _{A_{l}}(p,q),
\end{equation*}

Then for sufficiently large $r_{n}$ , we have%
\begin{equation}
T(r_{n},A_{j})\leq \exp _{p-1}\left\{ b_{1}\left( \log _{q-1}r_{n}\right)
^{\rho _{A_{l}}\left( p,q\right) }\right\} .  \label{2.13}
\end{equation}

Now for $r_{n}\in S_{2}\backslash S_{1},$ substituting $\left( \ref{2.11}%
\right) $ and $\left( \ref{2.12}\right) $ or $\left( \ref{2.13}\right) $
into $\left( \ref{2.10}\right) $ we obtain%
\begin{gather*}
\exp _{p-1}\left\{ b_{2}\left( \log _{q-1}r_{n}\right) ^{\rho _{A_{l}}\left(
p,q\right) }\right\} <T(r_{n},A_{l}) \\
\leq k\exp _{p-1}\left\{ b_{1}\left( \log _{q-1}r_{n}\right) ^{\rho
_{A_{l}}\left( p,q\right) }\right\} +o\left( \frac{\left( T\left(
r_{n}+2\left\vert c_{l}\right\vert ,f\right) \right) ^{1+\varepsilon }}{%
r_{n}^{\delta }}\right) .
\end{gather*}

The above implies%
\begin{equation*}
\left( 1-o(1)\right) \exp _{p-1}\left\{ b_{2}\left( \log _{q-1}r_{n}\right)
^{\rho _{A_{l}}\left( p,q\right) }\right\} <o\left( \frac{\left( T\left(
r_{n}+2\left\vert c_{l}\right\vert ,f\right) \right) ^{1+\varepsilon }}{%
r_{n}^{\delta }}\right) .
\end{equation*}

Hence the result follows.
\end{proof}

\bigskip Next we consider the properties of meromorphic solutions of $\left( %
\ref{1nh}\right) ~$where $A_{0}(z),A_{1}(z),....,A_{k}(z),F$ are entire
functions.

\begin{theorem}
\label{T2.4} Let $A_{0}(z),A_{1}(z),....,A_{k}(z)$ be entire functions that
satisfy the conditions stated in the Theorem $\ref{T2.3}$ and let $F$ be an
entire function. Then the followings hold

$\left( i\right) $ If $\rho _{F}\left( p,q\right) <\rho _{A_{l}}\left(
p,q\right) $ or $\rho _{F}\left( p,q\right) =\rho _{A_{l}}\left( p,q\right)
,\tau _{F}\left( p,q\right) <\tau _{A_{l}}\left( p,q\right) ,$ then every
meromorphic solution $\left( f\neq 0\right) $ of \ $\left( \ref{1nh}\right) $
satisfies $\rho _{f}\left( p,q\right) \geq \rho _{A_{l}}\left( p,q\right) .$

$\left( ii\right) $ If $\rho _{F}\left( p,q\right) >\rho _{A_{l}}\left(
p,q\right) ,$ then every meromorphic solution $\left( f\neq 0\right) $ of \ $%
\left( \ref{1nh}\right) $ satisfies $\rho _{f}\left( p,q\right) \geq \rho
_{F}\left( p,q\right) .$
\end{theorem}

\begin{proof}
We first consider Case $(i),$ when $\rho _{F}\left( p,q\right) <\rho
_{A_{l}}\left( p,q\right) $ or $\rho _{F}\left( p,q\right) =\rho
_{A_{l}}\left( p,q\right) ,\tau _{F}\left( p,q\right) <\tau _{A_{l}}\left(
p,q\right) .$

First let $f$ $\left( \neq 0\right) $ be a meromorphic solution of $\left( %
\ref{1nh}\right) $ and divide $\left( \ref{1nh}\right) $ by $f(z+c_{l})$ we
get%
\begin{equation*}
-A_{l}(z)=A_{k}(z)\frac{f(z+c_{k})}{f(z+c_{l})}+...+A_{l-1}(z)\frac{%
f(z+c_{l-1})}{f(z+c_{l})}+...+A_{1}(z)\frac{f(z+c_{1})}{f(z+c_{l})}+A_{0}(z)%
\frac{f(z)}{f(z+c_{l})}-\frac{F(z)}{f(z+c_{l})}.
\end{equation*}

The above expression can be written as%
\begin{equation}
-A_{l}(z)=\dsum\limits_{j=1,j\neq l}^{k}A_{j}(z)\frac{f(z+c_{j})}{f(z+c_{l})}%
+A_{0}(z)\frac{f(z)}{f(z+c_{l})}-\frac{F(z)}{f(z+c_{l})}.  \label{2.14}
\end{equation}

Now for any given $\varepsilon >0$ and sufficiently large $r,$ using Lemma $%
\left( \ref{L1.5A}\right) $ and Lemma $\left( \ref{L1.6}\right) $ in $\left( %
\ref{2.14}\right) $ we obtain (see \cite{1}) 
\begin{gather}
T(r,A_{l})=m(r,A_{l}(z))\leq m\left( r,\frac{F(z)}{f(z+c_{l})}\right)
+\dsum\limits_{j=0,j\neq l}^{k}m(r,A_{j}(z))  \notag \\
+\dsum\limits_{j=1,j\neq l}^{k}m\left( r,\frac{f(z+c_{j})}{f(z+c_{l})}%
\right) +m\left( r,\frac{f(z)}{f(z+c_{l})}\right) +O(1)  \notag \\
\leq T(r,F)+\dsum\limits_{j=0,j\neq l}^{k}T(r,A_{j}(z))+2T\left(
r+\left\vert c_{l}\right\vert ,f(z)\right) +o\left( \frac{\left( T\left(
r+2\left\vert c_{l}\right\vert ,f\right) \right) ^{1+\varepsilon }}{%
r^{\delta }}\right) ,  \label{2.15}
\end{gather}%
for $r\rightarrow \infty ,$ $r\notin S_{1},$ where $S_{1}$ is a set of
finite logarithmic measure.

Consider two real numbers $b_{1},b_{2}$ such that 
\begin{equation*}
\max \left\{ \tau _{A_{j}}(p,q):\rho _{A_{j}}\left( p,q\right) =\rho
_{A_{l}}\left( p,q\right) \right\} <b_{1}<b_{2}<\tau _{A_{l}}(p,q).
\end{equation*}

Now by Lemma $\left( \ref{L1.7}\right) $ there exists a subset $S_{2}\subset %
\left[ 1,+\infty \right) $ of infinite logarithmic measure such that for all 
$r\in S_{2},$we have 
\begin{equation*}
\log _{p-1}T\left( r,F\right) \leq b_{1}\left( \log _{q-1}r\right) ^{\rho
_{A_{l}}\left( p,q\right) }.
\end{equation*}

Therefore for a sequence $\left\{ r_{n}\right\} $ such that $r_{n}\in
S_{2},r_{n}\rightarrow \infty $ we have%
\begin{equation}
\log _{p-1}T(r_{n},F)\leq b_{1}\left( \log _{q-1}r_{n}\right) ^{\rho
_{A_{l}}\left( p,q\right) }.  \label{2.16}
\end{equation}

Now for $r_{n}\in S_{2}\backslash S_{1},$ substituting $\left( \ref{2.11}%
\right) $, $\left( \ref{2.12}\right) $ or $\left( \ref{2.13}\right) $ and $%
\left( \ref{2.16}\right) $ into $\left( \ref{2.15}\right) $ we obtain%
\begin{equation}
\exp _{p-1}\left\{ b_{2}\left( \log _{q-1}r_{n}\right) ^{\rho _{A_{l}}\left(
p,q\right) }\right\} <T(r_{n},A_{l})\leq (k+1)\exp _{p-1}\left\{ b_{1}\left(
\log _{q-1}r_{n}\right) ^{\rho _{A_{l}}\left( p,q\right) }\right\} +3\left(
T\left( 2r_{n},f\right) \right) ^{2}.  \label{2.17}
\end{equation}

By $\left( \ref{2.17}\right) $ first part of the theorem is proved.

Case $\left( ii\right) :$ Consider $\rho _{F}\left( p,q\right) >\rho
_{A_{l}}\left( p,q\right) $ and let let $f$ $\left( \neq 0\right) $ be a
meromorphic solution of $\left( \ref{1nh}\right) $.

Now for any given $\varepsilon >0$ and sufficiently large $r,$ using Lemma $%
\left( \ref{L1.5A}\right) $ and Lemma $\left( \ref{L1.6}\right) $ in $\left( %
\ref{2.14}\right) $ we obtain (see \cite{1})%
\begin{eqnarray}
T(r,F) &\leq
&\dsum\limits_{j=0}^{k}T(r,A_{j}(z))+\dsum\limits_{j=1}^{k}T\left(
r,f(z+c_{j})\right) +T(r,f\left( z\right) )+O(1)  \notag \\
&\leq &\dsum\limits_{j=0}^{k}T(r,A_{j}(z))+(2k+1)T(2r,f\left( z\right)
)+O(1).  \label{2.18}
\end{eqnarray}

Now by definition of the $(p,q)$-order there exists a sequence $\left\{
r_{n}\right\} $ such that $r_{n}\rightarrow \infty $ and for any given $%
\varepsilon $ $\left( 0<2\varepsilon <\rho _{F}\left( p,q\right) -\rho
_{A_{l}}\left( p,q\right) \right) ,$ we have%
\begin{equation}
T(r_{n},F)\geq \exp _{p-1}\left\{ \left( \log _{q-1}r_{n}\right) ^{\rho
_{F}\left( p,q\right) -\varepsilon }\right\}  \label{2.19}
\end{equation}%
and for $j=0,1,...,k$ 
\begin{equation}
T(r_{n},A_{j}(z))\leq \exp _{p-1}\left\{ \left( \log _{q-1}r_{n}\right)
^{b+\varepsilon }\right\} \leq \exp _{p-1}\left\{ \left( \log
_{q-1}r_{n}\right) ^{\rho _{A_{l}}\left( p,q\right) +\varepsilon }\right\} ,
\label{2.20}
\end{equation}%
where $b=\max \left\{ \rho _{A_{j}}\left( p,q\right) :j=0,1,...,k,\text{ }%
j\neq l\right\} <\rho _{A_{l}}\left( p,q\right) .$

Substituting $\left( \ref{2.19}\right) $ and $\left( \ref{2.20}\right) $
into $\left( \ref{2.18}\right) $ we obtain%
\begin{eqnarray}
\exp _{p-1}\left\{ \left( \log _{q-1}r_{n}\right) ^{\rho _{F}\left(
p,q\right) -\varepsilon }\right\} &\leq &(k+1)\exp _{p-1}\left\{ \left( \log
_{q-1}r_{n}\right) ^{\rho _{A_{l}}\left( p,q\right) +\varepsilon }\right\}
\label{2.21} \\
&&+(2k+1)T(2r,f\left( z\right) ).  \notag
\end{eqnarray}

By $\left( \ref{2.21}\right) $ second part of the theorem is proved.
\end{proof}

Next two theorems, i.e.\ Theorem \ref{C} and Theorem \ref{D} are based on
linear difference equation with meromorphic coefficients. In Theorem \ref{C}
we take homogeneous linear difference equation with one coefficient having
maximal $(p,q)-$order.

\begin{theorem}
\label{C}Let $A_{j}(z)(j=0,1,\ldots ,k)$ be meromorphic functions. \bigskip
If there exits an $A_{m}(z)(0\leq m\leq k)$ such that 
\begin{eqnarray*}
\lambda _{\frac{1}{A_{m}}}\left( p,q\right) &<&\rho _{A_{m}}\left(
p,q\right) <\infty , \\
\text{and }\max \{\rho _{A_{j}}\left( p,q\right) &:&j=0,1,\ldots ,k,~j\neq
m\}<\rho _{A_{m}}\left( p,q\right)
\end{eqnarray*}%
then for every nonzero meromorphic solution $f(z)$ of (\ref{1h}) satisfies $%
\rho _{f}\left( p,q\right) \geq \rho _{A_{j}}\left( p,q\right) .$
\end{theorem}

\begin{proof}
Let $f$ be a meromorphic solution of (\ref{1h}) and put $c_{0}=0$. We divide
(\ref{1h}) by $f(z+c_{m})~$and we get%
\begin{equation}
-A_{m}(z)=\sum_{j=0,j\neq m}^{k}\frac{A_{j}(z)f(z+c_{j})}{f(z+c_{m})}.
\label{c1}
\end{equation}

Now from Lemma \ref{L1.6} for any $\varepsilon >0$ we get,%
\begin{equation*}
m\left( r,\frac{f(z+c_{j})}{f(z+c_{m})}\right) \leq o\left( \frac{\left\{
T(r+3C,f)\right\} ^{1+\varepsilon }}{r^{\delta }}\right) ~,j=0,1,\ldots
,k,~j\neq m,
\end{equation*}%
where $C=\max_{0\leq \leq k}\{\left\vert c_{j}\right\vert :j=0,1,\ldots
,k\},~r\notin S_{1}$ where $S_{1}$ is chosen as Lemma \ref{L1.6}.

Using the above result, from (\ref{c1}) we get$,$ 
\begin{eqnarray}
T(r,A_{m}) &=&m(r,A_{m})+N(r,A_{m})  \notag \\
&\leq &\sum_{j=0,j\neq m}^{k}m(r,A_{j})+\sum_{j=0,j\neq m}^{k}m\left( r,%
\frac{f(z+c_{j})}{f(z+c_{m})}\right) +N(r,A_{m})+O(1)  \notag \\
&\leq &\sum_{j=0,j\neq m}^{k}T(r,A_{j})+o\left( \frac{\left\{
T(r+3C,f)\right\} ^{1+\varepsilon }}{r^{\delta }}\right) +N(r,A_{m})+O(1) 
\notag \\
&\leq &\sum_{j=0,j\neq m}^{k}T(r,A_{j})+\left\{ T(2r,f)\right\}
^{2}+N(r,A_{m})+O(1)  \label{c2}
\end{eqnarray}%
for $r\notin S_{1}.$

Now we denote,$~\rho =\rho _{A_{m}}\left( p,q\right) ,~\rho _{1}=\max \{\rho
_{A_{j}}\left( p,q\right) :j=0,1,\ldots ,k;~~j\neq m\}.~$Then clearly $\rho
_{1}<\rho $.

For that $\varepsilon ~$we have%
\begin{equation}
T(r,A_{m})>\exp _{p-1}\left\{ (\log _{q-1}r)^{\rho _{1}-\varepsilon }\right\}
\label{c3}
\end{equation}%
for sufficiently large $r$ with $r\in S_{2},~$where $S_{2}$ be a set with
infinite logarithmic measure.

And for $~j\neq m,~$for that $\varepsilon ~$we have%
\begin{equation}
T(r,A_{j})\leq \exp _{p-1}\left\{ (\log _{q-1}r)^{\rho _{1}+\varepsilon
}\right\}  \label{c4}
\end{equation}%
for sufficiently large $r.$

Again by the definition of $\lambda _{\frac{1}{A_{m}}}\left( p,q\right) ,$
we have for the above $\varepsilon $ and for sufficiently large $r$%
\begin{equation}
N(r,A_{m})\leq \exp _{p-1}\left\{ (\log _{q-1}r)^{\lambda +\varepsilon
}\right\}  \label{c5}
\end{equation}%
taking $\lambda _{\frac{1}{A_{m}}}\left( p,q\right) =\lambda .$

Now, using all the above relations (\ref{c3})-(\ref{c5}) and chosing $%
\varepsilon $ such that $0<\varepsilon <\frac{1}{2}\min \left\{ \rho -\rho
_{1},\rho -\lambda \right\} ,$ we have from (\ref{c2})%
\begin{eqnarray*}
\exp _{p-1}\left\{ (\tau -\varepsilon )(\log _{q-1}r)^{\rho }\right\}
&<&O\left( \exp _{p-1}\left\{ (\log _{q-1}r)^{\rho _{1}+\varepsilon
}\right\} \right) +3\left\{ T(2r,f)\right\} ^{2} \\
&&+\exp _{p-1}\left\{ (\log _{q-1}r)^{\lambda +\varepsilon }\right\} +O(1) \\
&\Rightarrow &3\left\{ T(2r,f)\right\} ^{2}>O\left( \exp _{p-1}\left\{ (\log
_{q-1}r)^{\rho +\varepsilon }\right\} \right)
\end{eqnarray*}%
for sufficiently large $r$ and $r\in S_{2}\backslash S_{1}.$

Which implies, $\rho _{f}\left( p,q\right) \geq \rho .$
\end{proof}

In the next theorem we consider non-homoheneous linear difference equation
which may have more than one coefficient with the maximal $(p,q)-$order. For
those type of equations we need to consider the $(p,q)-$type among the
coefficients having maximal $(p,q)-$order.

\begin{theorem}
\label{D}Let $A_{j}(z)(j=0,1,\ldots ,k)$ and $F(z)$ be meromorphic
functions. If there exits an $A_{m}(z)(0\leq m\leq k)$ such that 
\begin{eqnarray*}
\lambda _{\frac{1}{A_{m}}}\left( p,q\right) &<&\rho _{A_{m}}\left(
p,q\right) <\infty , \\
\max \{\rho _{A_{j}}\left( p,q\right) &:&j=0,1,\ldots ,k,~j\neq m\}\leq \rho
_{A_{m}}\left( p,q\right) \\
\text{and }\max \{\tau _{A_{j}}\left( p,q\right) &:&\rho _{A_{j}}\left(
p,q\right) =\rho _{A_{m}}\left( p,q\right) ,~j=0,1,\ldots ,k,~j\neq m\}<\tau
_{A_{m}}\left( p,q\right) <\infty ,
\end{eqnarray*}%
then the following cases arise:

i)\qquad If $\rho _{F}\left( p,q\right) <\rho _{A_{m}}\left( p,q\right) ,$ or%
$~\rho _{F}\left( p,q\right) =\rho _{A_{m}}\left( p,q\right) $ and $\tau
_{F}\left( p,q\right) \neq \tau _{A_{m}}\left( p,q\right) ,$ then every
nonzero meromorphic solution $f(z)$ of (\ref{1nh}) satisfies $\rho
_{f}\left( p,q\right) \geq \rho _{A_{m}}\left( p,q\right) .$

ii)\qquad If $\rho _{F}\left( p,q\right) >\rho _{A_{m}}\left( p,q\right) $
then every nonzero meromorphic solution $f(z)$ of (\ref{1nh}) satisfies $%
\rho _{f}\left( p,q\right) \geq \rho _{F}\left( p,q\right) .$
\end{theorem}

\begin{proof}
Let $f$ be a meromorphic solution of (\ref{1nh}). We divide (\ref{1nh}) by $%
f(z+c_{m})~$and we get%
\begin{equation}
-A_{m}(z)=\sum_{j=0,j\neq m}^{k}\frac{A_{j}(z)f(z+c_{j})}{f(z+c_{m})}-\frac{%
F(z)}{f(z+c_{m})}.  \label{d1}
\end{equation}

Now from Lemma \ref{L1.5A} and Lemma \ref{L1.6}, for any $\varepsilon >0$ we
get,%
\begin{eqnarray*}
m\left( r,\frac{f(z+c_{j})}{f(z+c_{m})}\right) &\leq &o\left( \frac{\left\{
T(r+3C,f)\right\} ^{1+\varepsilon }}{r^{\delta }}\right) ~,j=0,1,\ldots
,k,~j\neq m, \\
\text{and }m\left( r,\frac{1}{f(z+c_{m})}\right) &\leq &T\left( r,\frac{1}{%
f(z+c_{m})}\right) \\
&=&T\left( r,f(z+c_{m})\right) +O(1)\leq (1+O(1))T\left( r+C,f\right) .
\end{eqnarray*}%
where $C=\max_{0\leq \leq k}\{\left\vert c_{j}\right\vert :j=0,1,\ldots
,k\},~r\notin S_{1}$ where $S_{1}$ is chosen as Lemma \ref{L1.6}.

Using these from (\ref{d1}) we get$,$ 
\begin{eqnarray}
T(r,A_{m}) &=&m(r,A_{m})+N(r,A_{m})  \notag \\
&\leq &\sum_{j=0,j\neq m}^{k}m(r,A_{j})+\sum_{j=0,j\neq m}^{k}m\left( r,%
\frac{f(z+c_{j})}{f(z+c_{m})}\right) +m(r,F)  \notag \\
&&+m\left( r,\frac{1}{f(z+c_{m})}\right) +N(r,A_{m})+O(1)  \notag \\
&\leq &\sum_{j=0,j\neq m}^{k}T(r,A_{j})+o\left( \frac{\left\{
T(r+3C,f)\right\} ^{1+\varepsilon }}{r^{\delta }}\right) +T(r,F)  \notag \\
&&+(1+O(1))T\left( r+C,f\right) +N(r,A_{m})+O(1)  \notag \\
&\leq &\sum_{j=0,j\neq m}^{k}T(r,A_{j})+3\left\{ T(2r,f)\right\}
^{2}+T(r,F)+N(r,A_{m})+O(1)  \label{d2}
\end{eqnarray}%
for $r\notin S_{1}.$

Now denote,$~\rho =\rho _{A_{m}}\left( p,q\right) ,~\rho _{1}=\max \{\rho
_{A_{j}}\left( p,q\right) :j=0,1,\ldots ,k,~\rho _{A_{j}}\left( p,q\right)
<\rho \},~\tau =\tau _{A_{m}}\left( p,q\right) $ and $\tau _{1}=\max \{\tau
_{A_{j}}\left( p,q\right) :j=0,1,\ldots ,k,~j\neq m,~\rho _{A_{j}}\left(
p,q\right) =\rho \}.$

Then clearly%
\begin{equation*}
\rho _{1}<\rho \text{ and }\tau _{1}<\tau ~
\end{equation*}%
by the given hypothesis.

From Lemma \ref{1.7} for that $\varepsilon ~$we have%
\begin{equation}
T(r,A_{m})>\exp _{p-1}\left\{ (\tau -\varepsilon )(\log _{q-1}r)^{\rho
}\right\}  \label{d3}
\end{equation}%
for sufficiently large $r$ with $r\in S_{2},~$where $S_{2}$ is a set with
infinite logarithmic measure.

Again if for some $j,~\rho _{A_{j}}\left( p,q\right) <\rho ,$ then for that $%
\varepsilon ~$we have%
\begin{equation}
T(r,A_{j})\leq \exp _{p-1}\left\{ (\log _{q-1}r)^{\rho _{1}+\varepsilon
}\right\}  \label{d4}
\end{equation}%
for sufficiently large $r.$

And for those $j,~j\neq m$ and $\rho _{A_{j}}\left( p,q\right) =\rho ,$ for
that $\varepsilon ~$we have%
\begin{equation}
T(r,A_{j})\leq \exp _{p-1}\left\{ (\tau _{1}+\varepsilon )(\log
_{q-1}r)^{\rho }\right\}  \label{d5}
\end{equation}%
for sufficiently large $r.$

Again by the definition of $\lambda _{\frac{1}{A_{m}}}\left( p,q\right) ,$
we have for the above $\varepsilon $ and for sufficiently large $r$%
\begin{equation}
N(r,A_{m})\leq \exp _{p-1}\left\{ (\log _{q-1}r)^{\lambda +\varepsilon
}\right\}  \label{d6}
\end{equation}%
taking $\lambda _{\frac{1}{A_{m}}}\left( p,q\right) =\lambda .$

Now for the proof of the first part, we take $\rho _{F}\left( p,q\right)
<\rho ,~$then for that $\varepsilon ~$we have%
\begin{equation}
T(r,F)\leq \exp _{p-1}\left\{ (\log _{q-1}r)^{\rho _{F}\left( p,q\right)
+\varepsilon }\right\}  \label{d7}
\end{equation}%
for sufficiently large $r.$

Now, using all the above relations (\ref{d3})-(\ref{d7}) and chosing $%
\varepsilon $ such that $0<\varepsilon <\frac{1}{2}\min \left\{ \rho -\rho
_{1},\tau -\tau _{1},\rho -\lambda ,\rho -\rho _{F}\left( p,q\right)
\right\} ,$ we have from (\ref{d2})%
\begin{gather*}
\exp _{p-1}\left\{ (\tau -\varepsilon )(\log _{q-1}r)^{\rho }\right\}
<O\left( \exp _{p-1}\left\{ (\log _{q-1}r)^{\rho _{1}+\varepsilon }\right\}
\right) +O\left( \exp _{p-1}\left\{ (\tau _{1}+\varepsilon )(\log
_{q-1}r)^{\rho }\right\} \right) \\
+3\left\{ T(2r,f)\right\} ^{2}+\exp _{p-1}\left\{ (\log _{q-1}r)^{\rho
_{F}\left( p,q\right) +\varepsilon }\right\} +\exp _{p-1}\left\{ (\log
_{q-1}r)^{\lambda +\varepsilon }\right\} +O(1)
\end{gather*}%
\begin{equation*}
\Rightarrow 3\left\{ T(2r,f)\right\} ^{2}>O\left( \exp _{p-1}\left\{ (\log
_{q-1}r)^{\rho +\varepsilon }\right\} \right)
\end{equation*}%
for sufficiently large $r$ and $r\in S_{2}\backslash S_{1}.$

Which implies, $\rho _{f}\left( p,q\right) \geq \rho .$

Next we suppose that $\rho _{F}\left( p,q\right) =\rho $ and $\tau
_{F}\left( p,q\right) <\tau ,~$then for the above $\varepsilon ~$we have%
\begin{equation}
T(r,F)\leq \exp _{p-1}\left\{ (\tau _{F}\left( p,q\right) +\varepsilon
)(\log _{q-1}r)^{\rho }\right\}  \label{d8}
\end{equation}%
for sufficiently large $r.$

Now, using relations (\ref{d3})-(\ref{d6}) and (\ref{1nh}) and chosing $%
\varepsilon $ such that $0<\varepsilon <\frac{1}{2}\min \left\{ \rho -\rho
_{1},\tau -\tau _{1},\rho -\lambda ,\tau -\tau _{F}\left( p,q\right)
\right\} ,$ we have from (\ref{d2})%
\begin{eqnarray*}
\exp _{p-1}\left\{ (\tau -\varepsilon )(\log _{q-1}r)^{\rho }\right\}
&<&O\left( \exp _{p-1}\left\{ (\log _{q-1}r)^{\rho _{1}+\varepsilon
}\right\} \right) \\
&&+O\left( \exp _{p-1}\left\{ (\tau _{1}+\varepsilon )(\log _{q-1}r)^{\rho
}\right\} \right) +3\left\{ T(2r,f)\right\} ^{2} \\
&&+\exp _{p-1}\left\{ (\tau _{F}\left( p,q\right) +\varepsilon )(\log
_{q-1}r)^{\rho }\right\} +\exp _{p-1}\left\{ (\log _{q-1}r)^{\lambda
+\varepsilon }\right\} +O(1),
\end{eqnarray*}%
\begin{equation*}
\Rightarrow 3\left\{ T(2r,f)\right\} ^{2}>O\left( \exp _{p-1}\left\{ (\log
_{q-1}r)^{\rho +\varepsilon }\right\} \right)
\end{equation*}%
for sufficiently large $r$ and $r\in S_{2}\backslash S_{1}.$

Which implies, $\rho _{f}\left( p,q\right) \geq \rho .$

For the last case of first part we take, $\rho _{F}\left( p,q\right) =\rho $
and $\tau _{F}\left( p,q\right) >\tau ,$ then by Lemma \ref{1.7}, and for
the above $\varepsilon ~$we have%
\begin{equation}
T(r,F)>\exp _{p-1}\left\{ (\tau _{F}\left( p,q\right) -\varepsilon )(\log
_{q-1}r)^{\rho }\right\}  \label{d9}
\end{equation}%
for sufficiently large $r$ and $r\in S_{3},~$where $S_{3}$ is chosen as
Lemma \ref{1.7}.

Again by the definition of $\tau _{A_{m}}\left( p,q\right) $ $,$ we have for
the above $\varepsilon $ and for sufficiently large $r$%
\begin{equation}
T(r,A_{m})\leq \exp _{p-1}\left\{ (\tau +\varepsilon )(\log _{q-1}r)^{\rho
}\right\} .  \label{d10}
\end{equation}

Now from (\ref{1nh}) and Lemma \ref{L1.5A} it follows that%
\begin{equation}
T(r,F)\leq \sum_{j=0,j\neq m}^{k}T(r,A_{j})+T(r,A_{m})+(k+2)T(2r,f)
\label{d11}
\end{equation}%
for sufficiently large $r$

Now, using relations (\ref{d3}), (\ref{d4}) and (\ref{d9})-(\ref{d11}) and
choosing $\varepsilon $ such that $0<\varepsilon <\frac{1}{2}\min \left\{
\rho -\rho _{1},\tau -\tau _{1},\tau _{F}\left( p,q\right) -\tau \right\} ,$
we have 
\begin{eqnarray*}
\exp _{p-1}\left\{ (\tau _{F}\left( p,q\right) -\varepsilon )(\log
_{q-1}r)^{\rho }\right\} &<&O\left( \exp _{p-1}\left\{ (\log _{q-1}r)^{\rho
_{1}+\varepsilon }\right\} \right) +O\left( \exp _{p-1}\left\{ (\tau
_{1}+\varepsilon )(\log _{q-1}r)^{\rho }\right\} \right) \\
&&+O\left( \exp _{p-1}\left\{ (\tau +\varepsilon )(\log _{q-1}r)^{\rho
}\right\} \right) +(k+2)T(2r,f)
\end{eqnarray*}%
\begin{equation*}
\Rightarrow (k+2)T(2r,f)>O\left( \exp _{p-1}\left\{ (\log _{q-1}r)^{\rho
+\varepsilon }\right\} \right)
\end{equation*}%
for sufficiently large $r$ and $r\in S_{3}\backslash S_{1}.$

It follows that, $\rho _{f}\left( p,q\right) \geq \rho .$

For the second part of the theorem, we take $\rho _{F}\left( p,q\right)
>\rho =\rho _{A_{m}}\left( p,q\right) $.

If possible we suppose that $\rho _{f}\left( p,q\right) <\rho _{F}\left(
p,q\right) ,$ then from (\ref{1nh}) we get%
\begin{equation*}
\rho _{q}^{p}A_{k}(z)f(z+c_{k})+A_{k-1}(z)f(z+c_{k-1})+\cdots
+A_{1}(z)f(z+c_{1})+A_{0}(z)f(z)<\rho _{q}^{p}\left( F(z)\right)
\end{equation*}%
which is a contradiction.

Hence we have $\rho _{f}\left( p,q\right) \geq \rho _{F}\left( p,q\right) .$
\end{proof}

\end{document}